\documentclass{article}
\usepackage{color}
\definecolor{myurlcolor}{rgb}{0,0,0.5}
\usepackage{amsmath}
\usepackage{amsfonts}
\usepackage{amssymb}
\usepackage{amsthm}
\usepackage[all,2cell]{xy}
\UseAllTwocells
\usepackage{stmaryrd}
\usepackage{bbm}
\usepackage{enumerate}
\usepackage{accsupp}
\usepackage[title]{appendix}
\usepackage{comment}
\usepackage{hyperref}
\hypersetup{colorlinks,
linkcolor=black,
citecolor=black,
urlcolor=myurlcolor}

\newcommand{\from}{\colon} 
\newcommand{\cat}[1]{\mathcal{#1}} 
\newcommand{\scat}[1]{\mathbbm{#1}} 
\newcommand{\fcat}[1]{\mathbf{#1}} 
\newcommand{\demph}[1]{\textbf{#1}} 
\renewcommand{\epsilon}{\varepsilon} 
\newcommand{\such}{\mathrel{|}} 
\newcommand{\iso}{\cong} 
\newcommand{\bref}[1]{(\ref{#1})} 
\newcommand{\of}{\mathbin{\circ}} 
\DeclareMathOperator{\ev}{ev} 

\newcommand{\id}{\mathrm{id}}

\newcommand*\diff{\mathop{}\!\mathrm{d}}
\newcommand{\mnd}[1]{\mathbb{#1}}

\newcommand{\Set}{\fcat{Set}}

\newcommand{\Meas}{\fcat{Meas}}
\newcommand{\Conv}{\fcat{Conv}}

\newcommand{\psub}[1]{\mathbin{+_{#1}}}
\newcommand{\psubp}[1]{\mathbin{+'_{#1}}}
\newcommand{\const}[1]{\bar{#1}}
\newcommand{\reals}{\mathbb{R}}
\newcommand{\nat}{\mathbb{N}}

\newcommand{\aff}{\operatorname{aff}}
\newcommand{\intf}{\int_{\Omega}f \diff(-)}

\newtheorem{lem}[subsection]{Lemma}
\newtheorem{prop}[subsection]{Proposition}
\newtheorem{thm}[subsection]{Theorem}
\newtheorem{cor}[subsection]{Corollary}
\theoremstyle{definition}
\newtheorem{defn}[subsection]{Definition}

\title{Codensity and the Giry monad}
\author{Tom Avery\thanks{School of Mathematics, University of
 Edinburgh, Edinburgh EH9 3FD, United Kingdom.
Tom.Avery@ed.ac.uk}}
\date{}
\begin{document}
\sloppy

\maketitle

\begin{abstract}
The Giry monad on the category of measurable spaces sends
a space to a space of all probability measures on it. There is also a finitely additive
Giry monad in which probability measures are replaced by finitely additive probability
measures. We give a characterisation of both finitely and countably additive probability measures in
terms of integration operators giving a new description of the Giry monads. This is then used to show that the Giry
monads arise as the codensity monads of forgetful functors from certain categories
of convex sets and affine maps to the category of measurable spaces.
\end{abstract}

\section{Introduction}
\label{sec:intro}

In general there are many different probability measures on a given
measurable space, and the set of all of them can be made into a
measurable space in a canonical way. Thus we have a process which turns a space into a
new space whose points are the probability measures on the old one; this process
is described in categorical language by a monad.

On the other hand, there is a standard
categorical machine which turns a functor into a monad, namely the codensity monad
of the functor. We show that the monad described above is the output of this machine
when it is fed a natural forgetful functor involving certain convex sets. In other words, once we accept
the mathematical importance of these convex sets (which may be taken to be all
bounded, convex subsets of $\reals^n$ together with the set of sequences in the unit interval converging
to $0$), then the notion of a probability measure is categorically inevitable.

The monad sending a measurable space to its space of probability measures is called
the Giry monad, first defined in \cite{giry82}. There are many variations of this monad; 
in \cite{giry82} Giry defines both the monad mentioned above and a similar monad
on the category of Polish spaces. In this paper we will be mainly concerned with
Giry's monad on measurable spaces (which we refer to simply as the Giry monad),
and a modification in which probability measures are replaced by finitely additive
probability measures (the finitely additive Giry monad).

Note that there is a similar monad on $\Set$ that has sometimes been called the finitary
Giry monad \cite{fritz09} or the distribution monad \cite{jacobs11}.
It sends a set to the set of formal convex combinations of its elements, which can
be thought of as finitely supported probability measures. The algebras
for this monad are abstract ``convex spaces'', which have been independently 
discovered and investigated several times, for example in \cite{stone49}, \cite{gudder79}
and \cite{fritz09}. This \emph{finitary} Giry monad is not to be confused with the 
\emph{finitely additive} Giry monad, although they behave similarly on finite sets
(regarded as discrete measurable spaces).

The Kleisli category of the Giry monad has probability-theoretic
 significance \cite{panangaden99}; it is the category of measurable spaces and
``Markov kernels''. As a simple example, a finite set (with discrete $\sigma$-algebra)
equipped with an endomorphism in the Kleisli category of the Giry monad is precisely
a discrete time Markov chain. In \cite{doberkat06}, Doberkat shows that the 
Eilenberg--Moore category of (the Polish space version of) the Giry monad is the category of
continuous convex structures on Polish spaces with continuous affine maps.

The monads described above are examples of a loose family that we may think of
as ``measure monads''; in each instance, the monad sends a ``space'' to a space
of ``measures'' on it, where we must interpret space and measure appropriately.
Other examples include the ultrafilter monad (to which we shall return shortly), the
probabilistic powerdomain \cite{jonesPlotkin89}, the distribution monad \cite{jacobs11,
fritz09} and the monad defined by Lucyshyn-Wright in \cite{lucyshyn-wright12}. The idea of interpreting
monads measure-theoretically has been extensively pursued by Kock in \cite{kock12} 
and by Lucyshyn-Wright in \cite{lucyshyn-wright13}.

A common theme for all these monads is ``double dualisation''. For any notion of a
measure on a space $X$, there is a corresponding notion of integration.
Integration takes functions $f \colon X \to R$ from the space
to a set $R$ of scalars (usually the reals, positive reals or the unit interval), and returns
scalars in $R$. Such an integration operation can be thought of as an element of
\begin{equation*}
\label{eqn:double-dual}
\mathrm{Hom}(\mathrm{Hom}(X,R),R)
\end{equation*}
where the inner and outer $\mathrm{Hom}$'s must be interpreted appropriately in
different contexts. Thus notions of measure are closely related to double dualisation.

In some circumstances the measures can be completely characterised by their
integration operators. Perhaps the most well known instance of this phenomenon is the
Riesz--Markov--Kakutani representation theorem \cite{kakutani41}, which says that the
space of finite, signed, regular Borel measures on a compact Hausdorff space $X$ is isomorphic to
\[
\mathbf{NVS}(\mathbf{Top}(X,\mathbb{R}),\mathbb{R}),
\]
(as a normed vector space) where $\mathbf{Top}$ is the category of topological spaces and continuous maps, and
$\mathbf{NVS}$ is the category of normed vector spaces and bounded linear maps.
In Section~\ref{sec:int}, we give a similar (but easier) characterisation of probability measures
in terms of their integration operators, which is a correction of a claim of Sturtz \cite{sturtz14},
with many parts of the proof appearing there. Sturtz has since issued a corrected version
of his paper \cite{sturtz15}.

Such characterisations might make us hope that there is some general categorical machinery
for double dualisation that, when fed an appropriate and relatively simple input, naturally
gives rise to measure monads, and the Giry monad in particular. Codensity monads provide
such a categorical machine.

Codensity monads were first defined by Kock in \cite{kock66}, and the dual notion was studied independently
by Appelgate and Tierney in \cite{appelgateTierney69} under the name ``model-induced cotriple''.
Given a functor $U\colon \scat{C}\to \cat{M}$, the codensity monad of $U$ (when it exists)
is the right Kan extension $T^U$ of $U$ along itself. The universal property of Kan extensions
equips $T^U$ with a canonical monad structure. In \cite{leinster13}, Leinster describes how the
codensity monad can be thought of as a substitute for the monad induced by the
adjunction between $U$ and its left adjoint, even when the left adjoint does not
exist. In particular, when the left adjoint \emph{does} exist, the codensity monad
is the usual monad induced by the adjunction.

Codensity monads can be seen as a form of double dualisation via the end formula,
\[
T^U m = \int_{c \in \scat{C}}[\cat{M}(m,Uc),Uc].
\]
At first glance ``elements'' of this object would appear to be families of integration operators,
with the codomain of integration ranging over the objects of $\scat{C}$. However, in examples
of interest, such a family is determined by its component at a single object $i$, say, of $\scat{C}$.
The other objects serve to impose naturality conditions which force the $i$ component to preserve
certain algebraic structure which is encoded in the category $\scat{C}$. This idea will become
clear in the proof of Theorem~\ref{thm:giry-cod}.

As observed, for example, in \cite{leinster13}, the ultrafilter monad can be viewed as a measure monad
in the following way. An ultrafilter on a set $X$ consists of a set of subsets of $X$; thus
it can be thought of as a map from the power set of $X$ to $\{0,1\}$. Viewing $\{0,1\}$
as a subset of the unit interval $I$, it turns out that the functions
\[
2^X \to I
\]
corresponding to ultrafilters are precisely the finitely additive probability measures taking values in
$\{0,1\}$. This means that the ultrafilter monad is a primitive version of the finitely
additive Giry monad.

The ultrafilter monad is the codensity monad of the inclusion of the category of finite
sets into the category of sets; this was first proved by Kennison and Gildenhuys
 in \cite{kennisonGildenhuys71} and brought to wider attention by Leinster
in \cite{leinster13}. The main theorem of this paper (Theorem~\ref{thm:giry-cod}) is an analogous result for the Giry monads, with finite sets
replaced by certain convex sets, and with sets replaced by measurable spaces.

Despite the influence of Sturtz's work on the characterisation of probability meaures in
terms of integration operators mentioned above, our main result (Theorem~\ref{thm:giry-cod})
is substantially different from that of \cite{sturtz15}. Both seek to exhibit 
the Giry monad as the codensity monad of a particular functor, however the 
functors used differ in two significant respects: firstly, Sturtz uses the entire
category of convex spaces (as defined in \cite{fritz09}) as the domain of the functor,
whereas we will use a small subcategory of this; and secondly, Sturtz incorporates
an element of double dualisation into the functor itself, even before taking the codensity
monad, whereas we will use a more ``direct'' forgetful functor.



I am grateful to Tom Leinster for suggesting this topic to work on, for a lot of helpful
advice, and for many enlightening discussions.

\paragraph*{Conventions:}
We write $\Meas$ for the category of measurable spaces (i.e.\ sets equipped with a
$\sigma$-algebra of subsets) and measurable maps. We will often refer to
measurable spaces by their underlying sets, leaving the $\sigma$-algebra implicit.

We write $I$ for the unit interval $[0,1]$. When viewed as a measurable space,
we always equip it with the Borel $\sigma$-algebra. Given sets $B \subseteq A$,
we write $\chi_B \from A \to I$ for the characteristic function of $B$.
If $r \in I$, then $\bar{r} \from A \to I$ denotes the constant function with
value $r$.

If $A$ is a set and $m$ is an object of a category $\cat{M}$, then $[A,m]$ denotes
the $A$ power of $m$, that is, the product in $\cat{M}$ of $A$ copies of $m$. In
particular if $\cat{M} = \Set$, then $[A,m] = \Set(A,m)$, the set of functions from
 $A$ to $m$.

The integral sign $\int$ has two meanings in this paper: the occurences in Section~\ref{sec:cod},
and the single occurence in the introduction, represent the category theoretic notion
of an end (see X.5 in \cite{maclane71}). All other instances represent integration
with respect to a (possibly only finitely additive) probability measure.

\section{The Giry monads}
\label{sec:giry}

In this section we review some basic definitions relating to finitely additive probability
measures. We then define the finitely additive Giry monad, and the Giry monad as
a submonad.

Recall the following definitions.

\begin{defn}
Let $(\Omega,\Sigma)$ be a measurable space and $\pi \from \Sigma \to I$ (where $I$ is
the unit interval). Suppose
\begin{itemize}
\item $\pi (\Omega) = 1$, and
\item whenever $A,B \in \Sigma$ are disjoint, we have $\pi(A\cup B) = \pi(A) + \pi(B)$.
\end{itemize}
Then $\pi$ is called a \demph{finitely additive probability measure} on $(\Omega,\Sigma)$. 
Suppose additionally that,
\begin{itemize}
\item whenever $A_i \in \Sigma$ are pairwise disjoint for $i \in \nat$, we have
\[
\pi \left(\bigcup_{i=0}^{\infty}A_i \right) = \sum_{i = 0}^{\infty} \pi(A_i).
\]
\end{itemize}
Then $\pi$ is called a \demph{probability measure} on $(\Omega,\Sigma)$.
\end{defn}

The general theory of integration of finitely additive measures, as developed in
\cite{raoRao83}, is quite complex and subtle. There are several definitions of
integration; we will be concerned with the $D$-integral. However, we will only be interested in
integrating \emph{measurable, bounded} functions against finitely additive
\emph{probability} measures, which makes it possible to simplify the definition
considerably. Therefore we will for convenience briefly spell out how the integral is defined in
this special case.

Let $\pi$ be a finitely additive probability measure on $\Omega$. Recall that a function
$f \from \Omega \to \reals$ is \demph{simple} if it is a linear combination of
characteristic functions of measurable sets. The integral of a simple function is
defined by
\[
\int_{\Omega} \left(\sum_{i=1}^n a_i\chi_{A_i}\right) \diff \pi = \sum_{i=1}^n a_i \pi(A_i),
\]
and this does not depend on the choice of representation of the function.
For an arbitrary measurable, bounded, non-negative function $f \from \Omega \to \reals$, the integral is defined by
\[
\int_{\Omega} f \diff \pi = \sup \left\{ \int_{\Omega} f' \diff \pi \such f' \text{ is simple and } f' \leq f\right\} \in \reals,
\]
and this extends to functions that may take negative values in a standard way.
Note that the fact that $f$ is bounded guarantees that the supremum is finite.
The following lemma is easily verified.
\begin{lem}
\label{lem:int-approx}
Let $f \from \Omega \to \reals$ be measurable, bounded and non-negative, and $\pi$ a finitely additive
probability measure on $\Omega$. Then there is a sequence $f_n$
of simple functions converging uniformly to $f$, and for any such sequence
\[
\int_{\Omega} f_n \diff \pi \to \int_{\Omega} f \diff \pi.
\]
Moreover, $(f_n)_{n=1}^{\infty}$ can be taken to be a pointwise increasing (or decreasing) sequence. \qed
\end{lem}

Many basic results on integration against probability measures hold true for finitely additive
probability measures. In particular, integration is linear and order-preserving (4.4.13 (ii) and (vi)
 in \cite{raoRao83}), and the change of variables formula (Lemma~\ref{lem:change-var} below)
is valid. An important exception is that the monotone convergence theorem (and therefore
the dominated convergence theorem) does not hold for finitely additive measures. In fact,
the monotone convergence theorem holds if and only if the measure is countably additive.

We now move on to the definitions of the Giry monads.

\begin{defn}
Let $\Omega$ be a measurable space. Then $F\Omega$ is defined to be the 
set of finitely additive probability measures on $\Omega$, equipped
with the smallest $\sigma$-algebra such that
\begin{align*}
\ev_A \from & F\Omega \to  I \\
	  	& \pi  \mapsto  \pi (A)
\end{align*}
is measurable for each measurable $A \subseteq \Omega$. We write $G\Omega 
\subseteq F\Omega$ for the set of (countably additive) probability measures,
and equip it with the subspace $\sigma$-algebra.

Let $g \from \Omega \to \Omega'$ be measurable. Then
$Fg \from F\Omega \to F\Omega'$ is defined by
\[
Fg(\pi)(A') = \pi(g^{-1} (A'))
\]
for each measurable $A' \subseteq \Omega'$ and $\pi \in F\Omega$. We define $Gg \from G\Omega \to G\Omega'$ by restricting
$Fg$ to $G\Omega$.
\end{defn}

We call $Fg(\pi)$ the \demph{push-forward} of $\pi$ along $g$, written
as $g_*(\pi)$ by some authors. Integration for push-forward measures is described
by the \demph{change of variables formula}:
\begin{lem}
\label{lem:change-var}
Let $g \from \Omega \to \Omega'$ be measurable, $\pi \in F\Omega$ and
$f \from \Omega' \to \reals$ be measurable and bounded. Then
\[
\int_{\Omega}f \of g \diff \pi = \int_{\Omega'} f \diff Fg(\pi).
\]
\end{lem}
\begin{proof}
This is a familiar result for countably additive measures; see for example
Chapter~VIII Theorem~C in \cite{halmos50}. The proof for finitely additive
probability measures is identical.
\end{proof}

It is straightforward to check that the above definitions define functors
 $F,G\from \Meas \to \Meas$. The following lemma will be used to show that
the multiplication of each Giry monad is measurable, and also in Proposition~\ref{prop:bij-meas} below.

\begin{lem}
\label{lem:int-meas}
Let $f \from \Omega \to I$ be measurable. Then the map $\intf \from F\Omega \to I$ defined by
\[
\pi \mapsto \int_{\Omega}f \diff \pi
\]
is measurable.
\end{lem}
\begin{proof}
The inverse image of $[0,r]\subseteq I$ under this map is
\[
\left\{ \pi \such \int_{\Omega}f  \diff \pi \leq r \right\};
\]
we must show that this is measurable. Let $f_n$ be as in Lemma~\ref{lem:int-approx} and increasing.
Then the set above can be written as
\[
\bigcap_{n=1}^{\infty}\left\{ \pi \such \int_{F\Omega}f_n \diff \pi \leq r \right\}.
\]
In the case that $f = \chi_A$  for $A \subseteq \Omega$ measurable, $\intf = \ev_A$ so is
measurable by definition. Since integration is linear, and linear combinations of measurable functions
are measurable, $\intf$ is also measurable when $f$ is a simple function. Hence, returning to the
case of an arbitrary measurable $f$, each of the sets appearing in the above intersection
is measurable, and a countable intersection of measurable sets is measurable.
\end{proof}

We now describe the monad structure on $F$ and $G$.

\begin{defn}
Let $\Omega$ be a measurable space. The natural transformations
\[
\eta^{\mnd{F}} \from \id_{\Meas} \to F \quad \text{ and } \quad
\mu^{\mnd{F}} \from FF \to F
\]
are defined as follows. Let
\[
\eta^{\mnd{F}}_{\Omega}(\omega) (A) = \chi_A (\omega)
\]
where $\omega \in \Omega$ and $A \subseteq \Omega$ is measurable, so $\eta^{\mnd{F}}_{\Omega}(\omega)$ is the \emph{Dirac} or \emph{point
 measure} at $\omega$. Let
\[
\mu^{\mnd{F}}_{\Omega} (\rho)(A) = \int_{F\Omega} \ev_A \diff \rho.
\]
Here $\rho \in FF\Omega$ is a finitely additive probability measure on $F\Omega$, and
$A \subseteq \Omega$ is measurable, so in particular the map $\ev_A \from F\Omega \to I$
is measurable by definition. Thus integrating it against $\rho$ gives an element of $I$.
The natural transformations
\[
\eta^{\mnd{G}} \from \id_{\Meas} \to G \quad \text{ and } \quad
\mu^{\mnd{G}} \from GG \to G
\]
are defined similarly. It is easy to check that these formulae do define finitely (resp.\ countably)
additive probability measures on $\Omega$.
\end{defn}

Let us prove that $\mu_{\Omega}^{\mnd{F}}$ is measurable. If we take $f = \ev_A$ in
Lemma~\ref{lem:int-meas}, then $\intf$ is the composite
\[
\ev_A \of \mu_{\Omega}^{\mnd{F}} \from FF\Omega \to F\Omega \to I,
\]
hence this composite is measurable. Measurability of $\mu_{\Omega}^{\mnd{F}}$ follows
since the maps $\ev_A$ generate the $\sigma$-algebra on $F\Omega$. The proof
for $\mu_{\Omega}^{\mnd{G}}$ is similar and measurability of the units is obvious.

\begin{prop}
The above definitions give monads $\mnd{F} = (F,\eta^{\mnd{F}},
\mu^{\mnd{F}})$ and $\mnd{G} = (G,\eta^{\mnd{G}},\mu^{\mnd{G}})$
on $\Meas$.
\end{prop}
\begin{proof}
See \cite{giry82} for $\mnd{G}$. The proof for $\mnd{F}$ is similar. Note that Giry
invokes the monotone convergence theorem in the proof, however it can be replaced
by an instance of Lemma~\ref{lem:int-approx}.
\end{proof}

We call $\mnd{F}$ the \demph{finitely additive Giry monad} and $\mnd{G}$
the \demph{Giry monad}.

\section{Integration operators}
\label{sec:int}
We now turn to the characterisation of finitely and countably additive probability
measures in terms of integration operators. This will be used in
Section~\ref{sec:giry-cod} to characterise the Giry monads as codensity monads.

\begin{defn}
Let $\Omega$ be a measurable space, and let $\phi$ be a function
\[
\Meas (\Omega,I) \to I.
\]
We say that $\phi$ is a \demph{finitely additive integration operator} on $\Omega$ if,
\begin{itemize}
\item it is \demph{affine}: $\phi(rf + (1-r)g) =r \phi(f) + (1-r) \phi(g)$ for all $f,g \in \Meas(\Omega,I)$ and $r \in I$, and
\item it is \demph{weakly averaging}: $\phi (\const{r}) = r$ for all $r \in I$.
\end{itemize}
Recall $\const{r}$ denotes the constant function with value $r$. In \cite{sturtz14},
finitely additive integration operators
were called \demph{weakly averaging affine functionals}. We call $\phi$ an 
\demph{integration operator} (possibly with the qualification \demph{countably
additive} to avoid ambiguity) if, additionally,
\begin{itemize}
\item it \demph{respects limits}: if $f_n \in \Meas(\Omega,I)$ is
a sequence of measurable functions converging pointwise to $0$, then
$\phi(f_n)$ converges to $0$.
\end{itemize}
\end{defn}

\begin{defn}
\label{defn:int-func}
Let $\Omega$ be a measurable space. Write $S\Omega$ for the set of finitely additive integration operators
on $\Omega$ and $T\Omega$ for the set of integration operators. Equip $S\Omega$ with the smallest
$\sigma$-algebra such that
\begin{align*}
\ev_f \from & S\Omega \to  I \\
	  	& \phi  \mapsto  \phi (f)
\end{align*}
is measurable for each $f \in \Meas(\Omega,I)$, and define a $\sigma$-algebra on $S\Omega$
similarly.

Given $g \from \Omega \to \Omega'$ in $\Meas$, define $Sg \from S\Omega \to S\Omega'$ by
\[
Sg(\phi)(f) = \phi(f\of g)
\]
for $\phi \in S\Omega$ and $f \in \Meas(\Omega,I)$, and define $Tg$ similarly. This makes $S$ and $T$ functors $\Meas \to \Meas$.
\end{defn}

The following two lemmas show that (finitely additive) integration operators preserve
 more structure than the definition suggests, and they will be used often throughout
the rest of this paper.

\begin{lem}
\label{lem:int-prop}
Let $\phi \in S\Omega$, $f,f' \in \Meas(\Omega, I)$ and $r \in [0,\infty)$. Then
\begin{enumerate}[(i)]
\item if $rf \in \Meas(\Omega,I)$ then
$\phi (rf) = r\phi(f)$,
\item if  $f + f' \in \Meas(\Omega,I)$ then $\phi(f+f')
= \phi(f) + \phi(f')$, and
\item if $f \leq f'$ pointwise then $\phi(f) \leq \phi(f')$.
\end{enumerate}
\end{lem}
\begin{proof}
(i) follows from the affine property and the fact that $\phi(\const{0}) = 0$.

(ii) follows from the affine property and
\[
\phi(f+f') = 2\phi\left(\frac{1}{2}f + \frac{1}{2}f'\right),
\]
which is an instance of (i).

(iii) follows from (ii) applied to $f'= f+ (f'-f)$, using the fact that $f'-f \in \Meas(\Omega,I)$
(in particular $\phi(f'-f)$ is defined and is $\geq 0$).
\end{proof}

\begin{lem}
\label{lem:int-lim}
Let $\phi \in T\Omega$. If $f_n, f \in \Meas(\Omega,I)$ such that $f_n \to f$ pointwise,
then $\phi(f_n) \to \phi(f)$.
\end{lem}
\begin{proof}
Let 
\[
g_n(\omega) = \max(0,f_n(\omega) - f(\omega)) \quad \text{and} \quad h_n(\omega) = \max(0,f(\omega) - f_n(\omega)).
\]
Then $g_n, h_n \to 0$ pointwise, and
\[
f_n + h_n = g_n + f,
\]
so the result follows from part (ii) of the previous lemma and the fact that 
$ \phi(g_n)\to 0$ and $ \phi(h_n) \to 0$.
\end{proof}

The following lemma allows us to reduce propositions about finitely additive integration
operators to special cases involving only simple functions.

\begin{lem}
\label{lem:int-sup}
For any $\phi \in S\Omega$ and $f \in \Meas(\Omega, I)$, we have
\begin{align*}
\phi(f) &= \sup \{ \phi(g) \such g \in \Meas(\Omega,I) \textrm{ is simple and } g \leq f \} \\
&= \inf \{ \phi(g) \such g \in \Meas(\Omega,I) \textrm{ is simple and } g \geq f \}.
\end{align*}

\end{lem}
\begin{proof}
Write $L =  \{ \phi(g) \such g \in \Meas(\Omega,I) \text{ is simple and } g \leq f \}$. The inequality
$\phi(f) \geq \sup L$ follows from Lemma~\ref{lem:int-prop}~(iii). For the other inequality, fix
$\epsilon > 0$ and choose $g$ simple such that $g \leq f$ and $f - g \leq \const{\epsilon}$ (this is possible by
Lemma~\ref{lem:int-approx}). Then
\begin{align*}
\phi(f) &= \phi (g + (f-g) ) = \phi(g) + \phi(f-g) \\
&\leq \phi(g) + \phi(\const{\epsilon}) = \phi(g) + \epsilon \\
&\leq \sup L + \epsilon.
\end{align*}
Since $\epsilon>0$ was arbitrary, $\phi (f) \leq \sup L$. The other claim is proved similarly.

\end{proof}

In the next three propositions we establish isomorphisms of functors $F \cong S$
and $G \cong T$, and in the fourth we transfer the monad structure of $\mnd{F}$
and $\mnd{G}$ across these isomorphisms. Parts of their proofs are
due to Sturtz \cite{sturtz14}, however Sturtz incorrectly claims that $G \cong S$
rather than $F \cong S$, so we include the proofs here for clarity. Sturtz has since
issued a corrected version \cite{sturtz15}.

\begin{prop}
\label{prop:giry-int-bij}
Let $\Omega$ be a measurable space, and $\phi$ a finitely additive integration operator on
$\Omega$. Define $\Lambda(\phi)$ by
\[
\Lambda(\phi)(A) = \phi(\chi_A)
\]
for $A\subseteq \Omega$ measurable. Then
\begin{enumerate}[(i)]
\item $\Lambda$ is a bijection $S\Omega \iso F\Omega$, and
\item $\Lambda$ restricts to a bijection $\Xi \from T\Omega \iso G\Omega$.
\end{enumerate}
\end{prop}
\begin{proof}
(i) It is straightforward to check that $\Lambda(\phi)$ is a
finitely additive probability measure.

Given $\pi \in F\Omega$, define $\tilde{\Lambda}(\pi) \from \Meas(\Omega, I) \to I$ by
\[
\tilde{\Lambda}(\pi)(f) = \int_{\Omega} f \diff \pi
\]
for $f \in \Meas(\Omega,I)$. The affine and weakly averaging properties of $\tilde{\Lambda}(\pi)$ are standard
properties of integration against finitely additive measures (see 4.4.13 (ii) in \cite{raoRao83}),
so $\tilde{\Lambda}$ does define a function $F\Omega \to S\Omega$.

Now we must show that $\Lambda$ and $\tilde{\Lambda}$ are inverse to one another. In
one direction,
\[
\Lambda\tilde{\Lambda}(\pi)(A) = \int_{\Omega} \chi_A \diff \pi = \pi (A)
\]
for $ \pi \in F\Omega$ and $A \subseteq \Omega$ measurable, by definition of the integral.
In the other, first note that for $\phi \in S\Omega$ and $A \subseteq \Omega$ measurable,
\begin{align*}
\tilde{\Lambda}\Lambda (\phi)(\chi_A) = \int_{\Omega} \chi_A \diff \Lambda (\phi)
= \Lambda (\phi) (A) = \phi(\chi_A).
\end{align*}
It follows by Lemma~\ref{lem:int-prop} (i) and (ii) that $\tilde{\Lambda}\Lambda (\phi)(g) =  \phi(g)$
for all simple $g \in \Meas(\Omega,I)$. Then if $f \in \Meas(\Omega,I)$,
\begin{align*}
\phi (f) &= \sup \{ \phi(g) \such g \in \Meas(\Omega, I) \text{ is simple and } g \leq f \} \\
&= \sup \{ \tilde{\Lambda}\Lambda(\phi)(g) \such g \in \Meas(\Omega, I) \text{ is simple and } g \leq f \} \\
&= \tilde{\Lambda}\Lambda(\phi)(f),
\end{align*}
by Lemma~\ref{lem:int-sup}.


(ii) Suppose $\pi \in G\Omega$, and $f_n \to 0$ pointwise. Then
\[
\int_{\Omega} f_n \diff \pi \to 0
\]
by the dominated convergence theorem. So $\tilde{\Lambda}(\pi) \in T\Omega$.

Now suppose $\phi \in T\Omega$, and that $(A_n)_{n=1}^{\infty}$ is a disjoint family of measurable
sets. Write
\[
B_k = \bigcup_{n=k}^{\infty} A_k
\]
and $A = B_1$. Then $\chi_{B_k} \to 0$ pointwise, so $\Lambda(\phi)(B_k) \to 0$. But
\[
\Lambda(\phi)(A) = \sum_{n=1}^{k-1}\Lambda(\phi)(A_n) +\Lambda(\phi)(B_k)  \to \sum_{n=1}^{\infty}\Lambda(\phi)(A_n)
\]
as $k \to \infty$, so $\Lambda(\phi) \in G\Omega$ as required.
\end{proof}

\begin{prop}
\label{prop:bij-meas}
The bijections $\Lambda$ and $\tilde{\Lambda}$ are measurable (and hence so are $\Xi$
and $\tilde{\Xi}$).
\end{prop}

\begin{proof}
The $\sigma$-algebra on $F\Omega$ is generated by the maps
\[
\ev_A \from F\Omega \to I
\]
for measurable $A$, so in order to show that $\Lambda$ is measurable it is sufficient
to show that each $\ev_A \of \Lambda$ is measurable. But the diagram
\[
\xymatrix{
{S\Omega}\ar[r]^{\Lambda}\ar[dr]_{\ev_{\chi_A}} & {F\Omega}\ar[d]^{\ev_A} \\
& I
}
\]
commutes, and $\ev_{\chi_A} \from S\Omega \to I$ is measurable by definition of the
$\sigma$-algebra on $T\Omega$.

To show that $\tilde{\Lambda}$ is measurable we must show that $\ev_f \of \tilde{\Lambda}$ is
measurable for each $f \in \Meas(\Omega,I)$. But
\[
\ev_f \of \tilde{\Lambda}(\pi) = \int_{\Omega} f \diff \pi,
\]
So the composite is measurable by Lemma~\ref{lem:int-meas}.
\end{proof}

\begin{prop}
The maps $\Lambda \from S\Omega \to F\Omega$ and $\Xi \from T\Omega \to G\Omega$ are natural in $\Omega$.
\end{prop}
\begin{proof}
A straightforward verification, or see Theorem~4.4 in \cite{sturtz14}.
\end{proof}

Thus $\Lambda$ and $\Xi$ are natural isomorphisms $S \iso F$ and $T \iso G$.
Since $F$ and $G$ carry monad structures, there are unique monad structures
on $S$ and $T$ making $\Lambda$ and $\Xi$ into morphisms of monads, giving
an alternative description of the Giry monads.

\begin{prop}
\label{prop:giry-int-mnd}
The monad structure $\mnd{S} = (S,\eta^{\mnd{S}},\mu^{\mnd{S}})$ on $S$
induced by $\Lambda$ is given by
\[
\eta_{\Omega}^{\mnd{S}}(\omega)(f) = f(\omega)
\]
for $\Omega \in \Meas$, $\omega \in \Omega$ and $f \in \Meas(\Omega,I)$, and
\[
\mu_{\Omega}^{\mnd{S}}(\psi)(f) = \psi (\ev_f),
\]
for $\psi \in SS\Omega$. Similarly for $\Xi \from G \iso T$.
\end{prop}
The second of these expressions deserves some explanation. Here $\psi \in SS\Omega$
is an affine and weakly averaging function
\[
\psi \from \Meas(S\Omega,I) \to I.
\]
The elements of $S\Omega$ are functions $\Meas(\Omega,I) \to I$ and $f\in \Meas(\Omega,I)$,
so we have $\ev_f \from S\Omega \to I$, and this is measurable. Therefore $\psi$
 can be applied to $\ev_f$, yielding an element of $I$.

\begin{proof}[Proof of Proposition~\ref{prop:giry-int-mnd}.]
We know that the unit and multiplication of the induced monad structure on $S$ make the diagrams
\[
\xymatrix{
{S\Omega}\ar[rr]^{\Lambda_{\Omega}} & & {F\Omega} \\
& {\Omega}\ar[ul]^{\eta_{\Omega}^{\mnd{S}}}\ar[ur]_{\eta_{\Omega}^{\mnd{F}}} &
}
\]
and
\[
\xymatrix{
{SS\Omega}\ar[r]^{\Lambda_{S\Omega}}\ar[d]_{\mu_{\Omega}^{\mnd{S}}} & {FS\Omega}\ar[r]^{F\Lambda_{\Omega}} & {FF\Omega}\ar[d]^{\mu_{\Omega}^{\mnd{F}}} \\
{S\Omega}\ar[rr]_{\Lambda_{\Omega}} & & {F\Omega}
}
\]
commute. Therefore we have $\eta_{\Omega}^{\mnd{S}} = \tilde{\Lambda}_{\Omega} \of \eta_{\Omega}^{\mnd{F}}$
and $\mu_{\Omega}^{\mnd{S}} = \tilde{\Lambda}_{\Omega} \of \mu_{\Omega}^{\mnd{F}} \of F\Lambda_{\Omega} \of \Lambda_{S\Omega}$.

If $\omega \in \Omega$ and $ f \in \Meas(\Omega,I)$ then
\begin{align*}
\eta_{\Omega}^{\mnd{S}}(\omega)(f) &= \tilde{\Lambda}_{\Omega} \of \eta_{\Omega}^{\mnd{F}}(\omega)(f) \\
& = \int_{\Omega} f \diff  (\eta_{\Omega}^{\mnd{F}}(\omega)) \\
& = f(\omega).
\end{align*}

Now suppose $\psi \in SS\Omega$. Then, if $A \subseteq \Omega$ is measurable, we have
\begin{align*}
\mu_{\Omega}^{\mnd{S}} (\psi)(\chi_A) & = \mu_{\Omega}^{\mnd{F}} \of F\Lambda_{\Omega} \of\Lambda_{S\Omega} (\psi)(A) \\
&= \int_{F\Omega} \ev_A\diff (F\Lambda_{\Omega} \of\Lambda_{S\Omega}(\psi)) \\
&= \int_{S\Omega} \ev_A \of \Lambda_{\Omega}\diff (\Lambda_{S\Omega}(\psi))\\
&= \int_{S\Omega} \ev_{\chi_A} \diff (\Lambda_{S\Omega}(\psi)) \\
&= \tilde{\Lambda}_{S\Omega} \of\Lambda_{S\Omega} (\psi)(\ev_{\chi_A}) \\
&= \psi (\ev_{\chi_A}).
\end{align*}
Note that  if $g = \sum_{i=1}^n a_i \chi_{A_i}$ is a simple function in $\Meas(\Omega,I)$ then $\ev_g = \sum_{i=1}^{n} a_i \ev_{\chi_{A_i}}$ as elements
of $\Meas(S\Omega, I)$, so
\[
\mu_{\Omega}^{\mnd{S}} (\psi)(g) =  \sum_{i=1}^n a_i \mu_{\Omega}^{\mnd{S}} (\psi)(\chi_{A_i}) =  \sum_{i=1}^n a_i \psi(\ev_{\chi_{A_i}}) = \psi(\ev_g).
\]
Now if $f \in \Meas(\Omega,I)$, we have
\begin{align*}
\mu_{\Omega}^{\mnd{S}} (\psi)(f) &= \sup \{ \mu_{\Omega}^{\mnd{S}} (\psi)(g) \such g \in \Meas(\Omega,I) \textrm{ is simple and } g \leq f \} \\
& =  \sup \{  \psi(\ev_g) \such g \in \Meas(\Omega,I) \textrm{ is simple and } g \leq f \}\\
& \leq \psi(\ev_f)\\
&\leq \inf \{ \psi(\ev_g) \such g \in \Meas(\Omega,I) \textrm{ is simple and } g \geq f \}\\
&= \inf \{ \mu_{\Omega}^{\mnd{S}} (\psi)(g) \such g \in \Meas(\Omega,I) \textrm{ is simple and } g \geq f \}\\
&= \mu_{\Omega}^{\mnd{S}} (\psi)(f),
\end{align*}
where the first and last equalities are Lemma~\ref{lem:int-sup}, and the inequalities are due to the facts that if $f \leq g$ then $\ev_f \leq \ev_g$ and
that $\psi$ is order-preserving. Hence $\mu_{\Omega}^{\mnd{S}} (\psi)(f) = \psi(\ev_f)$ as required. The proof for $\Xi$ is similar.


%

\end{proof}

\section{Review of codensity monads}
\label{sec:cod}

In this section we review the basics of codensity monads. A more thorough
introduction can be found in \cite{leinster13}. The main purpose of this section,
besides a review of the definitions, is to obtain a description of codensity monads
that will make it easy to establish an isomorphism of monads between the codensity
monads defined in Section~\ref{sec:giry-cod} and the monads defined in Section~\ref{sec:int}.
 This description is given by Equations \bref{eqn:gen-elt-func}, \bref{eqn:gen-elt-unit}
and \bref{eqn:gen-elt-mult}.

Let $\scat{C}$ be a small category, $\cat{M}$ a complete, locally small category, and
$U \from \scat{C} \to \cat{M}$ a functor. Then the right
Kan extension of $U$ along itself always exists; it consists of a functor
$T^U \from \cat{M} \to \cat{M}$ and a natural transformation
$\kappa^U \from T^U U \to U$, which are defined by the following universal property:
if $H \from \cat{M} \to \cat{M}$ and $\lambda \from HU \to U$, then there is a
unique $\tau \from H \to T^U$ such that
\[
\vcenter{
\xymatrix@=50pt{
{\scat{C}}\ar[r]^{U}\ar[dr]_U\druppertwocell\omit{<-3.3>\kappa^U} & {\cat{M}}\dtwocell^H_{<1.2>T^U}{\tau} \\
& {\cat{M}}
}}
\quad = \quad
\vcenter{
\xymatrix@=50pt{
{\scat{C}}\ar[r]^U\ar[dr]_U\druppertwocell\omit{<-3.3>\lambda} & {\cat{M}}\ar[d]^H \\
& {\cat{M}}.
}}
\]
We make $T^U$ the endofunctor part of a monad
 $\mnd{T}^U = (T^U,\eta^U,\mu^U)$, called the \demph{codensity monad
of $G$}, defining $\eta^U$ and $\mu^U$ using the universal property
of $T^U$ as follows:
\[
\vcenter{
\xymatrix@=50pt{
{\scat{C}}\ar[r]^{U}\ar[dr]_U\druppertwocell\omit{<-3.3>\kappa^U} & {\cat{M}}\dtwocell^{<2>\id_{\cat{M}}}_{<1.2>T^U}{\eta^U} \\
& {\cat{M}}
}}
\quad = \quad
\vcenter{
\xymatrix@=50pt{
{\scat{C}}\ar[r]^U\druppertwocell\omit{=<-3.3>}\ar[dr]_U & {\cat{M}}\ar@{=}[d] \\
& {\cat{M}}
}}
\]
and
\[
\vcenter{
\xymatrix@C=20pt@R=40pt{
{\scat{C}}\ar[ddrr]_U\ar[rr]^U\ddrruppertwocell\omit{<-3>\kappa^U}  & & {\cat{M}}\ar[dd]|{T^U}\ar[dr]^{T^U}\dduppertwocell\omit{<-3>\mu^U} &\\
 & & & {\cat{M}}\ar[dl]^{T^U} \\
 & &{\cat{M}} &
}}
\quad = \quad
\vcenter{
\xymatrix@C=70pt@R=40pt{
{\scat{C}}\ar[r]^U\ar[dr]_U\ar[ddr]_U\druppertwocell\omit{<-2.5> \; \kappa^U}\ddruppertwocell\omit{<-2.5>\kappa^U} & {\cat{M}}\ar[d]^{T^U} \\
& {\cat{M}}\ar[d]^{T^U} \\
&{\cat{M}.}
}}
\]
The fact that these maps satisfy the monad axioms follows from the uniqueness
part of the universal property.

It will also be useful to have a more explicit description of $\eta^U$
and $\mu^U$. The end formula for right Kan extensions (\cite{maclane71} X.4) gives
\[
T^Um \iso \int_{c \in \scat{C}} [\cat{M}(m,Uc),Uc].
\]
Let
\[
\ev_f \from  \int_{c \in \scat{C}} [\cat{M}(m,Uc),Uc] \to Uc
\]
be the canonical limit projection (where $f \from m \to Uc$ in $\cat{M}$).

If $g \from m \to m'$, then $T^Ug$ is defined to be the unique morphism making
\[
\xymatrix{
{T^Um}\ar[dr]^{\ev_{f\of g}}\ar[d]_{T^Ug} & \\
{T^U m'}\ar[r]_{\ev_f} & Uc
}
\]
commute for each $f \from m' \to Uc$.

We will now describe the functor and monad structure of $T^U$ in terms of
generalised elements. Recall that if $m \in \cat{M}$, a \demph{generalised element $e$
with shape $s \in \cat{M}$}, or \demph{$s$-element}, of $m$ is simply a morphism
$e \from s \to m$, and we write
\[
e \in_s m.
\]
The shape $s$ has also been called the 
\demph{stage of definition} of $e$, for example in \cite{kock06}.

Any morphism $g \from m \to m'$ defines a function (also denoted $g$) mapping $s$-elements
 of $m$ to $s$-elements of $m'$:
if $e \in_s m$, then
\[
g(e) = g \of e \in_s m'.
\]
Furthermore, a consequence of the Yoneda lemma is that any such function
defined on generalised elements corresponds to a unique morphism $m \to m'$
(provided it is natural in $s$). This provides a convenient way of describing
morphisms in $\cat{M}$. Note that
\begin{align*}
\cat{M}(s,T^Um) &\iso \cat{M}\left(s, \int_{c \in \scat{C}} [\cat{M}(m,Uc),Uc] \right)\\
&\iso \int_{c\in \scat{C}}[\cat{M}(m,Uc),\cat{M}(s,Uc)],
\end{align*}
so, by the nature of limits in $\Set$, an $s$-element of $T^Um$ can be
thought of as a family of functions (natural in $c$) that map morphisms $m \to Uc$
to $s$-elements of $Uc$. Given $\alpha \in_s T^Um$, and
$f \from m \to Uc$,
\[
\ev_f (\alpha) = \alpha_c (f) \in_s Uc.
\]
Thus, in terms of generalised elements, the functor $T^U$ is defined by
\begin{equation}
\label{eqn:gen-elt-func}
((T^Ug)(\alpha))_c(f) = \ev_f \of (T^U g) (\alpha) = \ev_{f\of g} (\alpha) = \alpha_c(f\of g),
\end{equation}
where $g \from m \to m'$ and $f \from m' \to Uc$.
In \cite{kock66}, Kock describes $\eta^U$ and $\mu^U$ in terms of the equations
\[
\ev_f \of\eta^U_m = f \quad \text{and} \quad \ev_f \of \mu^U_m = \ev_{\ev_f}
\]
for each $c \in \scat{C}$ and $f\in \cat{M}(m,Uc)$. Translating these into generalised element notation,
$\eta^U$ and $\mu^U$ are defined by
\begin{equation}
\label{eqn:gen-elt-unit}
(\eta^U_m (e))_c(f) = \ev_f \of \eta^U_m (e) =  f(e)
\end{equation}
and
\begin{equation}
\label{eqn:gen-elt-mult}
(\mu^U_m (\beta))_c(f) = \ev_f \of \mu^U_m(\beta) =\ev_{\ev_f}(\beta) =  \beta_c(\ev_f)
\end{equation}
where $e \in_s m$ and $\beta \in_s T^UT^Um$.

\section{Probability measures via codensity}
\label{sec:giry-cod}

We will now show that the finitely additive Giry monad and the Giry monad arise
as codensity monads.

\begin{defn}
A \demph{convex set} is a convex subset of a real vector space. That is, $c \subseteq V$ is
a convex set if for all $x,y \in c$ and $r \in I$ we have $rx + (1-r)y \in c$. We write
\[
x \psub{r} y = rx + (1-r)y.
\]

If $c,c'$ are convex sets, then an \demph{affine map} $h \from c \to c'$ is a function such that
\[
h(x \psub{r} y) = h(x) \psub{r} h(y)
\]
for all $x,y \in c$ and $r \in I$.
\end{defn}

There is a more abstract notion of a convex space, investigated in \cite{fritz09}, namely
an algebra for the distribution monad mentioned in the introduction.
These more general convex spaces are used by Sturtz in \cite{sturtz14}.
However, all the convex spaces we will be concerned with are convex subsets of vector
spaces, so we omit the more general definition.

We choose the term ``affine map" rather than ``convex map"  to avoid confusion
with the notion of a ``convex function" (a real-valued
function with convex epigraph). This is potentially ambiguous: the term affine is already
used for a map between vector spaces that preserves \emph{affine} combinations
(i.e.\ linear combinations of the form $rx + (1-r)y$ where $r \in \reals$) rather than
just \emph{convex} combinations (those for which $r \in I$). However it is easily seen
that a map preserving convex combinations also preserves whatever affine combinations
exist in the domain.
Moreover, we have the following useful result:

\begin{lem}
\label{lem:aff-ext}
Let $c$ and $c'$ be convex subsets of real vector spaces $V$ and $V'$, and let $h \from
c \to c'$ be an affine map. Then $h$ has a unique affine extension $\aff(c) \to \aff(c')$,
where
\[
\aff(c) = \{ru + (1-r)v \such u, v \in c \text{ and } r \in \reals \}
\]
is the \demph{affine span} of $c$ in $V$.
\end{lem}

\begin{proof}
Define
\[
h(ru + (1-r)v) = rh(u) + (1-r)h(v).
\]
It is straightforward to check that this is well-defined and affine.
\end{proof}

The following corollary will be used in the proof of Proposition~\ref{prop:conv-bound}.

\begin{cor}
\label{cor:lin-ext}
Let $\phi \from \Meas(\Omega,I) \to I$ be a finitely additive integration operator.
Then $\phi$ has a unique linear extension
\[
\phi \from \Meas_b (\Omega,\reals) \to \reals,
\]
where $\Meas_b(\Omega,\reals)$ denotes the vector space of bounded measurable
maps $\Omega \to \reals$.
\end{cor}
\begin{proof}
Regarding $\Meas(\Omega,I)$ as a convex set in $\Meas(\Omega,\reals)$, we have
\[
\aff(\Meas(\Omega,I)) = \Meas_b(\Omega,\reals),
\]
and $\aff(I) = \reals$ as a subset of $\reals$, so there is a unique affine extension by the
previous lemma. Moreover, since $\phi$ preserves $0$, the extension is in fact linear.
\end{proof}

The domain categories of the functors whose codensity monads we will prove to be the Giry monads
are both full subcategories of the category of convex sets.

\begin{defn}
\begin{enumerate}[(i)]
\item Let $c_0$ be the vector space of real sequences converging to $0$, and let $d_0
\subseteq c_0$ be the (convex) set of sequences in $c_0$ contained entirely in $I$. 
We will occasionally mention the $\sup$-norm on $c_0$ defined by
\[
\|x\|_{\infty} = \sup_n |x_n|.
\]
\item Let $\scat{C}$ be the category whose objects are all finite powers of $I$ (including
$1 = I^0$) and all affine maps between them.
\item Let $\scat{D}$ be the category whose objects are all finite powers of $I$, together
with $d_0$, and all affine maps between them.
\end{enumerate}
\end{defn}

\begin{prop}
\label{prop:conv-maps}
\begin{enumerate}[(i)]
\item
Every affine $h \from I^n \to I$ is of the form
\[
h(x) = a_0 + \sum_{i=1}^n a_i x_i
\]
for some $a_i \in \reals$.
\item
Every affine $ h \from d_0 \to I$ is of the form
\[
h(x) = a_0 + \sum_{i=1}^{\infty} a_i x_i
\]
for some $a_i \in \reals$ with $\sum_{i=1}^{\infty}|a_i| < \infty$.
\end{enumerate}
\end{prop}
\begin{proof}
(i) By Lemma~\ref{lem:aff-ext}, there is a unique extension of $h$ to an affine map $\reals^n \to \reals$. But any
affine map between vector spaces can be written as a linear map followed by a translation
of the codomain. The general form of such a map $\reals^n \to \reals$ is as claimed.

(ii) As in (i), $h$ has a unique affine extension $c_0 \to \reals$ (since $\aff(d_0) = c_0$),
and this can be written as a linear map followed by a translation; write $h'$ for the linear part. We claim that
$h'$ is continuous with respect to the $\sup$-norm on $c_0$:

For subsets $A$ and $B$ of a vector space, write
\[
A-B = \{a-b \such a \in A, b \in B\}
\]
Then
\[
h'(d_0 - d_0) = h'(d_0) - h'(d_0) \subseteq I- I = [-1,1],
\]
but $d_0 - d_0$ is the unit ball in $c_0$ with the $\sup$-norm, and $h'$ maps it into a
bounded set, so $h'$ is continuous.

But a continuous linear functional on $c_0$ is of the form
\[
x \mapsto\sum_{i=1}^{\infty} a_i x_i
\]
for some $a_i \in \reals$ with  $\sum_{i=1}^{\infty}|a_i| < \infty$ (this fact is a common
exercise in courses on functional analysis; see for example Exercise~1 in Chapter~3 of 
\cite{bollobas99}). So $h$ is as claimed.
\end{proof}

Every object of $\scat{D}$ can be given a measurable space structure as a
subspace of a product of copies of $I$ (recall that $I$ is always given the Borel
$\sigma$-algebra).

\begin{prop}
\label{prop:conv-meas}
All the maps in $\scat{D}$ are measurable.
\end{prop}
\begin{proof}
A map $c \to c'$ in $\scat{D}$ is measurable if and only if its composite with each projection
$c' \to I$ is measurable, so it is sufficient to show that affine maps $h \from c \to I$ are measurable.

If $c= I^n$, then $h$ is of the form described in Proposition~\ref{prop:conv-maps} (i), and
is measurable, since all the basic arithmetic operations are.

Suppose $c = d_0$ and $h$ is of the form described in Proposition~\ref{prop:conv-maps} (ii).

Now, consider the topology on $d_0$ as a subset of $c_0$ with the $\sup$-norm.
A basic open set for this topology is of the form
\[
U = d_0 \cap \prod_{i=1}^{\infty} (x_i - \epsilon, x_i + \epsilon)
\]
for some $x \in d_0$ and $\epsilon > 0$. The $\sigma$-algebra on $d_0$ is generated by sets of the form $\prod_{i=1}^{\infty} A_i$,
where $A_i = I$ for all but one $i$, say $i_0$, and $A_{i_0}$ is measurable. Clearly a basic open set
can be written as a countable intersection of such sets, so is measurable. On the other hand, $d_0$ is a separable metric
space (a countable dense set is given by the sequences of rationals that are eventually $0$), and therefore
second countable, by a standard exercise in topology (e.g. Exercise~2.23 in~\cite{rudin53}). Moreover, the countable base we obtain is
contained in the original base, and so consists of measurable sets. Hence every open subset is a countable
union of measurable sets, so is measurable, and it follows that a norm-continuous
function $d_0 \to I$ is measurable. But $h$ is the composite of a continuous linear functional
on $c_0$ and a translation of $\reals$, so is continuous, and hence measurable.
\end{proof}

\begin{cor}
There are natural forgetful functors $U \from \scat{C} \to \Meas$ and $V \from \scat{D} \to \Meas$.
\qed
\end{cor}

We can now state the main theorem of this paper:

\begin{thm}
\label{thm:giry-cod}
\begin{enumerate}[(i)]
\item The codensity monad $\mnd{T}^U$ of $U \from \scat{C} \to \Meas$ is isomorphic to the
finitely additive Giry monad.
\item The codensity monad $\mnd{T}^V$ of $V \from \scat{D} \to \Meas$ is isomorphic to the
Giry monad.
\end{enumerate}
\end{thm}

The proof will follow shortly, but first let us make some general observations about the measurable
space $T^U \Omega$. We saw in Section~\ref{sec:cod} that an $s$-element of $T^U \Omega$ is a
family $\alpha$ of functions
\[
\alpha_c \from \Meas(\Omega,Uc) \to \Meas(s,Uc),
\]
natural in $c$. In particular, an \emph{ordinary} element of $T^U \Omega$
(which is the same as a generalised element of shape $1$) is a natural family of functions
\[
\alpha_c \from \Meas(\Omega,Uc) \to Uc.
\]
The $\sigma$-algebra on $T^U \Omega$ is the smallest such that
\[
\ev_f \from T^U \Omega \to Uc
\]
is measurable, for each $f \in \Meas(\Omega,Uc)$.

\begin{lem}
\label{lem:nat-proj}
Let $\alpha \in T^U \Omega$. Then
\[
\alpha_{I^n} = (\alpha_I)^n \from \Meas(\Omega,I^n) \iso \Meas(\Omega,I)^n \to I^n.
\]
The same is true if $\alpha \in T^V \Omega$, and then $\alpha_{d_0}$ is also obtained by
applying $\alpha_I$ componentwise. 
\end{lem}
\begin{proof}
Let
\[
f = (f_1,\ldots, f_n) \in \Meas(\Omega,I^n)
\]
By commutativity of
\[
\xymatrix{
{\Meas(\Omega,I^n)}\ar[r]^-{\alpha_{I^n}}\ar[d]_{(\pi_i)_*} & I^n\ar[d]^{\pi_i} \\
\Meas(\Omega,I)\ar[r]_-{\alpha_I} & I
}
\]
we have $(\alpha_{I^n}(f))_i = \alpha_I(f_i)$, as required. The proof for $d_0$ is
similar.
\end{proof}

\begin{proof}[Proof of Theorem~\ref{thm:giry-cod}.]
(i) We will establish a bijection between $T^U \Omega$ and $S \Omega$, where $S$ is as in Definition 
\ref{defn:int-func}. Given $\alpha \in T^U \Omega$ we claim that
\[
\alpha_I \from \Meas(\Omega,I) \to I,
\]
is affine and weakly averaging, i.e.\ an element of $S\Omega$. Suppose $f,g \in \Meas(\Omega,I)$,
and $r \in I$. The map 
\[
\psub{r} \from I^2 \to I
\]
is affine, so
\[
\xymatrix{
{\Meas(\Omega,I^2)}\ar[r]^-{\alpha^2}\ar[d]_{(\psub{r})_*} & {I^2}\ar[d]^{\psub{r}} \\
{\Meas(\Omega,I)}\ar[r]_-{\alpha_I} & I
}
\]
commutes, and following $(f,g)$ around this diagram yields
\[
\alpha_I(f) \psub{r} \alpha_I(g) = \alpha_I(f \psub{r} g),
\]
so $\alpha_I$ is affine. The fact that for $r \in I$, we have $\alpha_I(\const{r}) = r$
follows from commutativity of
\[
\xymatrix{
{\Meas(\Omega,1)}\ar[r]\ar[d]_{r_*} & 1\ar[d]^{r} \\
{\Meas(\Omega,I)}\ar[r]_-{\alpha_I} & I.
}
\]
So $\alpha_I$ is a finitely additive integration operator. Now suppose $\phi \in S\Omega$ is a
finitely additive integration operator. Define
\[
\alpha_{I^n} \from \Meas(\Omega,I^n) \to I^n
\]
by
\[
\alpha_{I^n}(f_1,\ldots,f_n) = (\phi(f_1),\ldots, \phi(f_n)).
\]
We must check that this is natural with respect to all maps in $\scat{C}$. Since any
function into $I^n$ is determined by its composites with the projections, it is sufficient
to check naturality with respect to maps with codomain $I$. Suppose $h \from I^n \to I$
is of the form
\[
h(x) = a_0 + \sum_{i=1}^n a_i x_i
\]
from Proposition~\ref{prop:conv-maps}. Then if $f \in \Meas(\Omega,I^n)$,
\begin{align*}
h\of \alpha_{I^n} (f) &= h (\phi(f_1),\ldots,\phi(f_n)) \\
& =a_0 + \sum_{i=1}^n a_i \phi(f_i) \\
&= \phi\left(a_0 + \sum_{i=1}^n a_i f_i \right) \\
&= \phi (h \of f) \\
&= \alpha_I \of h_* (f)
\end{align*}
as required (where the third equality comes from implicitly identifying $\phi$ with
its linear extension from Corollary~\ref{cor:lin-ext}, and the weakly averaging property). It is immediate
that these assignments
\[
T^U \Omega \to S\Omega \quad \text{ and } \quad S\Omega \to T^U \Omega
\]
are inverse to each other, and measurable.

To see that these bijections are natural and respect the monad structures on $T^U$ and
$S$, we must establish the commutativity of certain diagrams. Recall that the functor
and monad structures of $S$ are defined in Definition \ref{defn:int-func} and Proposition~\ref{prop:giry-int-mnd}
respectively. A description of the relevant structure on $T^U$ is given by Equations
\bref{eqn:gen-elt-func}, \bref{eqn:gen-elt-unit} and \bref{eqn:gen-elt-mult} of Section~\ref{sec:cod}
 with $s=1$, so that these become statements about ordinary, rather than
generalised elements. From these facts, and recalling that the (unnamed) bijection
$T^U \Omega \iso S\Omega$ is given by sending $\alpha \in T^U \Omega$ to $\alpha_I$, it is
straightforward to check that the relevant diagrams commute.

(ii) Let $\alpha \in T^V \Omega$. As before, $\alpha_I$ is affine and weakly averaging; now
we show it respects limits. Suppose $f_n \from \Omega \to I$ is a sequence of measurable
functions converging pointwise to $0$. Then $f$ defines an element of $\Meas(\Omega,d_0)$.
By Lemma~\ref{lem:nat-proj}, 
\[
(\alpha_{d_0}(f))_i = \alpha_I(f_i),
\]
so since $\alpha_{d_0}(f) \in d_0$, we must have $\alpha_I(f_i) \to 0$.

Now suppose $\phi$ is an integration operator. Let $\alpha_{I^n}$ be defined as in
(i), and define $\alpha_{d_0}(f)_i = \phi(f_i)$ for $f \in \Meas(\Omega,d_0)$. The fact that $\phi$ preserves limits of
sequences converging to $0$ means that $\alpha_{d_0}$ does map into $d_0$.
Once again we only need to check that $\alpha$ is natural with respect to maps with
codomain $I$, and for maps out of $I^n$ this is as before.

Suppose $h \from d_0 \to I$ is affine, say
\[
h(x) = a_0 + \sum_{i=1}^{\infty}a_i x_i.
\]
Then
\begin{align*}
h \of \alpha_{d_0} (f) &= h\left( (\phi(f_i))_{i=1}^{\infty}\right) \\
&= a_0 + \sum_{i=1}^{\infty} a_i\phi(f_i) \\
&= \lim_{N \to \infty}\left( a_0 + \sum_{i=1}^N a_i\phi(f_i)\right) \\
&= \lim_{N \to \infty} \phi\left(a_0 + \sum_{i=1}^N a_i f_i \right) \\
&= \phi\left( a_0 + \sum_{i=1}^{\infty} a_i f_i \right) \\
&= \phi(h\of f) \\
&= \alpha_I \of h_* (f),
\end{align*}
as required. Here we have again implicitly used the linear extension of $\phi$ from
Corollary~\ref{cor:lin-ext}, and also the result that $\phi$ preserves all limits (Lemma~\ref{lem:int-lim}).
As in (i), the remainder of the proof is a series of straightforward checks.
\end{proof}
Note that in the preceding proof we only made use of the objects $1$, $I$,
$I^2$, and in part (ii), $d_0$. Thus we could have taken $\scat{C}$ and $\scat{D}$ to
be the categories with just these objects and affine maps between them. In fact, even more is true:

\begin{prop}
\label{prop:giry-cod-mon}
Let $M$ and $N$ be the monoids of affine endomorphisms of $I^2$ and $d_0$
respectively. Then
\begin{enumerate}[(i)]
\item The codensity monad $\mnd{T}^M$ of the action of $M$ on $I^2$ in $\Meas$ is the finitely
additive Giry monad.
\item The codensity monad $\mnd{T}^N$ of the action of $N$ on $d_0$ in $\Meas$ is the Giry
monad.
\end{enumerate}
\end{prop}
Recall that an action of a monoid on an object of a category is essentially the same as a functor
from the monoid (regarded as a category with one object) to the category, so it makes sense to
talk about the codensity monad of an action.
\begin{proof}[Proof of Proposition~\ref{prop:giry-cod-mon}.]
We will prove (ii); (i) is similar. It is clear from Theorem~\ref{thm:giry-cod} that an
integration operator on $\Omega$ will define an element of $T^N \Omega$. Given $\alpha \in T^N \Omega$,
which we regard as a function $\Meas(\Omega,d_0) \to d_0$ that commutes with affine endomorphisms
of $d_0$, we must construct an integration operator. Define
\begin{align*}
\iota_0 \from 1 \to d_0, \quad & \text{an arbitrary map,} \\
\iota_1 \from I \to d_0, \quad & x \mapsto (x,0,0,\ldots), \\
\iota_2 \from I^2 \to d_0, \quad & (x_1,x_2) \mapsto (x_1,x_2, 0, \ldots), \\
\pi'_i \from d_0 \to d_0, \quad & (x_1,x_2, \ldots) \mapsto (x_i,0,\ldots), \\
\psubp{r} \from d_0 \to d_0, \quad & (x_1, x_2, \ldots) \mapsto (x_1 \psub{r} x_2, 0 \ldots), \\
r' \from d_0 \to d_0, \quad & (x_1,x_2,\ldots) \mapsto (r, 0, \ldots)
\end{align*}
(where $r \in I$), and let $\phi$ be the composite
\[
\xymatrix{
{\Meas(\Omega,I)}\ar[r]^-{(\iota_1)_*} & {\Meas(\Omega,d_0)}\ar[r]^-{\alpha} & {d_0}\ar[r]^{\pi_1} & {I.}
}
\]
Then $\alpha$ is obtained by applying $\phi$ componentwise, by the commutativity of
\[
\xymatrix{
{\Meas(\Omega,d_0)}\ar@{=}[r]\ar[d]_{(\pi_i)_*} & {\Meas(\Omega,d_0)}\ar[r]^-{\alpha}\ar[d]_{(\pi'_i)_*} & {d_0}\ar@{=}[r]\ar[d]_{\pi'_i} & {d_0}\ar[d]^{\pi_i} \\
{\Meas(\Omega,I)}\ar[r]_-{(\iota_1)_*} & {\Meas(\Omega,d_0)}\ar[r]_-{\alpha} & {d_0}\ar[r]_{\pi_1} & I,
}
\]
for each $i$. In particular, since $\alpha (f) \in d_0$, it follows that $\phi$ respects limits.
The affine and weakly averaging properties of $\phi$ follow from the commutativity of
\[
\xymatrix{
{\Meas(\Omega,I^2)}\ar[r]_{(\iota_2)_*}\ar[d]_{(\psub{r})_*}\ar@/^10pt/@<3pt>[rrr]^{\phi^2} & {\Meas(\Omega,d_0)}\ar[r]_-{\alpha}\ar[d]_{(\psubp{r})_*} & {d_0}\ar[r]_{(\pi_1,\pi_2)}\ar[d]_{\psubp{r}} & {I^2}\ar[d]^{\psub{r}} \\
{\Meas(\Omega,I)}\ar[r]^{(\iota_1)_*}\ar@/_10pt/@<-3pt>[rrr]_{\phi} & {\Meas(\Omega,d_0)}\ar[r]^-{\alpha} & {d_0}\ar[r]^{\pi_1} & I
}
\]
and
\[
\xymatrix{
{\Meas(\Omega,1)}\ar[r]_{(\iota_0)_*}\ar[d]_{r_*}\ar@/^10pt/@<3pt>[rrr]^{\phi^0} & {\Meas(\Omega,d_0)}\ar[r]_-{\alpha}\ar[d]_{r'_*} & {d_0}\ar[r]\ar[d]_{r'} & 1\ar[d]^r \\
{\Meas(\Omega,I)}\ar[r]^{(\iota_1)_*}\ar@/_10pt/@<-3pt>[rrr]_{\phi} & {\Meas(\Omega,d_0)}\ar[r]^-{\alpha} & {d_0}\ar[r]^{\pi_1} & I
}
\]
respectively.
\end{proof}
The preceding proposition gives categories of convex sets that are in some
sense minimal (although not uniquely so) such that the 
codensity monads of their inclusions into $\Meas$ are the Giry monads. It is
natural to ask how large a category of convex sets (or even convex spaces
in the sense of \cite{fritz09}) can be and still give rise
to the Giry monad. We have not answered this question precisely, but the
following proposition at least gives a class of convex sets that can be included
in the domain category without altering the codensity monad.

\begin{prop}
\label{prop:conv-bound}
Let $\scat{C}'$ be the category of compact, convex subsets of $\reals^n$
(where $n$ can vary) with affine maps between them and let $\scat{D}'$ be
 similar but with $d_0$ adjoined. Then the codensity monads of
the forgetful functors $U'\from \scat{C}' \to \Meas$ and $V' \from \scat{D}'
\to \Meas$ are the finitely additive Giry monad and the Giry monad respectively.
\end{prop}

\begin{proof}
An element of $T^{U'}\Omega$ is a family of functions $\Meas(\Omega,U'c') \to U'c'$ natural in
$c' \in \scat{C}'$. Since $\scat{C} \subseteq \scat{C}'$, and $U$ is the restriction of $U'$, such
a family restricts to a family $\Meas(\Omega,Uc) \to Uc$ natural in $c \in \scat{C}$, that is,
 an element of $T^{U} \Omega\cong S\Omega$. Therefore we just
have to check that every element of $T^{U} \Omega$ has a unique extension to an element
of $T^{U'}\Omega$. Similarly for $V$ and $V'$.

Suppose $\phi$ is a finitely additive integration operator on $\Omega$ and $c$ a compact
convex subset of $\reals^n$. Write $\phi$ also for the unique linear extension of
 $\phi$ to
\[
\Meas_b(\Omega,\reals) \to \reals,
\]
which exists by Corollary~\ref{cor:lin-ext}.
We can define
\[
\alpha_c \from \Meas(\Omega,c) \to \reals^n
\]
by applying $\phi$ in each coordinate. We will now show that
\begin{enumerate}[(i)]
\item If $h \from \reals^n \to \reals^m$ is affine then
\[
\xymatrix{
{\Meas(\Omega,c)}\ar[r]^-{\alpha_c}\ar[d]_{h_*} & {\reals^n}\ar[d]^h \\
{\Meas(\Omega,h(c))}\ar[r]_-{\alpha_{h(c)}} & {\reals^m}
}
\]
commutes, and
\item If $f \from \Omega \to c$ is measurable then $\alpha_c(f) \in c$ (this is presumably
known but we were unable to find a reference).
\end{enumerate}
(i) Since $\reals^m$ is a power of $\reals$ it is sufficient to consider $h\from \reals^n
\to \reals$. Such an $h$ is of the form
\[
(x_1,\ldots,x_n) \mapsto a_0 + a_1 x_1 + \ldots a_n x_n,
\]
and the fact that $\alpha$ commutes with such maps follows from linearity and the
weakly averaging property of $\phi$.

(ii) Let $f \in \Meas(\Omega,c)$, and suppose
for a contradiction that $\alpha_c(f) \notin c$. By applying an affine change of
coordinates, which we may do without loss of generality using (i), we may assume that
$\alpha_c(f) = 0 \notin c$. Then by
the separating hyperplane theorem (see for example Corollary~2.4 in Chapter~3 of \cite{mangasarian94})
 there is a linear functional $h \from \reals^n \to \reals$ and $\epsilon > 0$ such that 
$h(x) > \epsilon$ for all $x \in c$. By (i), we have
\[
\phi (h \of f) = h (\alpha_c (f)) = 0
\]
But $h \of f > \const{\epsilon}$, and so, since $\phi$ is order-preserving and weakly averaging,
$\phi (h \of f) \geq \epsilon$. This is a contradiction, completing the proof of (ii).


From (ii), we have maps $\alpha_c \from \Meas(\Omega,c) \to c$, all that remains is to check that they
commute with all affine maps $c \to c'$. As usual, since $c'$ is a subset of a power of $\reals$,
it is sufficient to check commutativity of all diagrams
\[
\xymatrix{
{\Meas(\Omega,c)}\ar[r]^-{\alpha_c}\ar[d]_{h_*} & {c}\ar[d]^h \\
{\Meas_b(\Omega,\reals)}\ar[r]_-{\phi} & {\reals}
}
\]
for affine $h \from c \to \reals$, where $c \subseteq \reals^n$ is compact and convex. By Lemma~\ref{lem:aff-ext}, $h$ has an affine
extension $\aff(c) \to \reals$ which we shall also write as $h$; we will extend this to an affine map $\reals^n \to \reals$
as follows. Choose $x_0 \in \aff(c)$, and write $L = \{x \in \reals^n \such x+ x_0 \in \aff(c)\}$.
Then $L $ is a \emph{linear} subspace of $\reals^n$, and the map $l \from L \to \reals$
defined by
\[
l(x) = h(x+x_0)-h(x_0)
\]
is linear, so $l$ has a linear extension $l' \from \reals^n \to \reals$. Let
\[
h'(x) = l'(x-x_0) + h(x_0)
\]
for $x \in \reals^n$; then $h'$ is the desired affine extension of $h$. The result
follows by the same argument as in (i) above, with $h'$ in place of $h$.
\end{proof}

\small
\bibliography{bibliography}

\end{document}